\documentclass{article}
\usepackage{amssymb,amsmath,amsthm}
\usepackage{graphicx}
\usepackage{array}
\usepackage{xcolor}
\usepackage[normalem]{ulem}
\usepackage{hyperref}

\theoremstyle{plain}
\newtheorem{thm}{Theorem}[section]

\theoremstyle{definition}
\newtheorem{definition}[thm]{Definition}
\newtheorem{remark}[thm]{Remark}

\DeclareMathOperator{\CA}{\!\mathit{CA}}

\title{Arndt and Carlitz Compositions}
\author{Brian Hopkins and Aram Tangboonduangjit}
\date{}

\begin{document}
\maketitle

\begin{abstract}
Carlitz considered integer compositions in which adjacent parts must be unequal.  Arndt recently initiated the study of restricted compositions based on conditions applied to certain pairs of parts rather than to individual parts.  Here, we combine and generalize these notions, establishing enumeration results using both combinatorial proofs and generating functions.  Motivations for our generalizations include the gap-free compositions studied by Hitczenko and Knopfmacher and the Rogers--Ramanujan integer partitions.


\end{abstract}

\section{Background}

Given a positive integer $n$, the compositions of $n$ are all ordered sequences of positive integers $(c_1, \ldots, c_t)$ called parts whose sum is $n$.  Write $C(n)$ for the compositions of $n$ and let $c(n) = | C(n)|$; in general, we will use capital letters for sets and the corresponding lower case letters for cardinalities.   Although analysis of integer compositions dates back at least to ancient India \cite{s86}, formal study and consistent terminology began with MacMahon \cite{m93}; he proved, for example, that $c(n) = 2^{n-1}$.  It is convenient to set $c(0) = 1$ for the ``empty composition.''

For purposes of motivation, we will also mention integer partitions which differ in that the order of the parts is disregarded.  Write $P(n)$ for the partitions of $n$.  For example, $C(3) = \{(3), (2,1), (1,2), (1,1,1)\}$ while $P(3) = \{(3), (2,1), (1,1,1)\}$ following the convention that partition parts are listed in nonincreasing order.  

One of the more engaging concepts in the theory of compositions involves the local condition $c_i \ne c_{i+1}$ for each $i$ introduced by Carlitz \cite{c76}; the ensuing research includes \cite{ch07, gh02, k05, kp98, w11}.  Carlitz's condition can be understood as a composition version of partitions with distinct parts, considered by Euler.  Another motivation for our work is gap-free compositions, where $\mid c_i - c_{i+1} \mid \le 1$ for each $i$, studied by Hitczenko and Knopfmacher \cite{hk05}.  We note that both Carlitz and gap-free compositions are treated by the general framework of Bender and Canfield \cite{bc05}.

The other primary ingredient of this project is the program following Arndt of considering restrictions, for each $i$, on pairs of parts $(c_{2i-1}, c_{2i})$ with no restriction between $c_{2i}$ and $c_{2i+1}$.  This began with Arndt's 2013 observation that applying the restriction $c_{2i-1} >  c_{2i}$ gives compositions of $n$ counted by the Fibonacci numbers; the current authors published the first proof of this in 2022 \cite{ht22}.  Variations of Arndt's original restriction and other generalizations have followed \cite{cr24, hm25, ht23, nacr25, p23}. 

Here, we combine the notions of Carlitz and Arndt compositions (Section 2) and consider two natural generalizations (Sections 3 and 4).  Our methods are primarily combinatorial, but we also determine related generating functions.

We note that Prodinger considered a more complicated combination of the Carlitz and Arndt conditions in 2023 \cite{p23} prompting us to use the notation $\CA(n)$ for the compositions studied here to avoid any confusion with his $AC(n)$.

\section{Combining the Carlitz and Arndt conditions}

Our first objects of study are the compositions of $n$ satisfying a combination of the Carlitz and Arndt restrictions.

\begin{definition} \label{cadef}
Let the Carlitz--Arndt compositions $\CA(n)$ be the compositions of $n$ such that $c_{2i-1} \ne c_{2i}$ for each $i$.  If the number of parts is odd, then there is no restriction on the last part.
\end{definition}

See Table \ref{tabca} for examples of Carlitz--Arndt compositions.  Notice, for instance, that $(2,1,1,2) \in \CA(6)$ even though $c_2 = c_3 = 1$ as the composition satisfies $c_1 \ne c_2$ and $c_3 \ne c_4$.

\begin{table}[h]
\begin{center}
\begin{tabular}{c|l|c}
$n$ & $\CA(n)$ & $ca(n)$ \\ \hline
1 & 1 & 1 \\
2 & 2 & 1 \\
3 & 3, 21, 12 & 3 \\
4 & 4, 31, 211, 13, 121  & 5 \\
5 & 5, 41, 32, 311, 23, 212, 14, 131, 122 & 9 \\
6 & 6, 51, 42, 411, 321, 312, 24, 231, 213, 2121,  \\ & \quad 2112, 15, 141, 132, 123, 1221, 1212 & 17
\end{tabular}
\end{center}
\caption{Carlitz--Arndt compositions and their counts for small values of $n$ using the shorthand $2112$ for $(2,1,1,2)$, for example.} \label{tabca}
\end{table}

The sequence $ca(n)$ matches \cite[A000213]{o}, a version of ``tribonacci'' numbers.  We establish the connection in our first theorem.

\begin{thm}\label{k=1}
The Carlitz--Arndt compositions $\CA(n)$ satisfy \[ca(n) = ca(n-1) + ca(n-2) + ca(n-3)\] with initial values $ca(1) = ca(2) = 1$ and $ca(3) = 3$.
\end{thm}

\begin{proof}
The initial values follow from Table \ref{tabca}.

For the recurrence, we establish a slightly stronger result based on the parity of the length of compositions in $\CA(n)$.  Let $\CA^o(n)$ be the compositions of $\CA(n)$ with an odd number of parts, similarly $\CA^e(n)$ for an even number of parts.  We demonstrate bijections
\begin{gather}
 \CA^o(n) \cong \CA(n-1), \label{k=1odd} \\
 \CA^e(n) \cong \CA(n-2) \cup \CA(n-3). \label{k=1even}
 \end{gather}
Since $\CA^o(n) \cap \CA^e(n) = \varnothing$ and $\CA(n) = \CA^o(n) \cup \CA^e(n)$, these will imply 
\[ \CA(n) \cong \CA(n-1)  \cup \CA(n-2) \cup \CA(n-3)\]
from which the result follows.

\textbf{First map, from $\CA^o(n)$:} Given $(c_1, \ldots, c_{2s-1}) \in \CA^o(n)$, modify the last part by
\[c_{2s-1} \mapsto c_{2s-1}-1\]
which we understand to mean removing the last part if $c_{2s-1} = 1$.  This does not change any of the pairwise conditions before $c_{2s-1}$, so the resulting composition is in $\CA(n-1)$.

\textbf{First map, from $\CA(n-1)$:} Given $(c_1, \ldots, c_t) \in \CA(n-1)$, modify the last part by
\[ c_t \mapsto
\begin{cases} c_t + 1 & \text{if $t$ is odd}, \\
(c_t, 1) & \text{if $t$ is even}. \end{cases} \]
In either case, the result is clearly a composition of $n$ of odd length.  In the first case, $\CA(n)$ has no restriction on the last part of an odd-length composition.  In the second case, the pair ending in $c_t$ already satisfied the conditions for $\CA(n-1)$ and there is no issue with the last part 1.  So the image is in $\CA^o(n)$.

It is clear that these maps are inverses, establishing \eqref{k=1odd}.

\textbf{Second map, from $\CA^e(n)$:} Given $(c_1, \ldots, c_{2s}) \in \CA^e(n)$, modify the last two parts by
\[ (c_{2s-1}, c_{2s}) \mapsto
\begin{cases} c_{2s} - 2 & \text{if $c_{2s-1} = 1$ \; (image in $\CA(n-3)$),} \\
(c_{2s-1} - 1, c_{2s} -1) & \text{if $c_{2s-1} > 1$ \; (image in $\CA(n-2)$).} \end{cases} \]
where, again, zeros are removed.  In the first case, $c_{2s-1} = 1$ implies that $c_{2s} \ge 2$ so that $c_{2s}-2$ is nonnegative and the sum has decreased by three (one from $c_{2s-1}=1$ and two from $c_{2s}$).  Any pairwise conditions before $c_{2s-1}$ are not changed and there is at most one more part, so the image is indeed in $\CA(n-3)$.  In the second case, $c_{2s-1} > 1$ means that $c_{2s-1} - 1$ is positive and not equal to $c_{2s} -1$ (which could be zero) since $c_{2s-1} \ne c_{2s}$, so the image is in fact in $\CA(n-2)$.

\textbf{Second map, from $\CA(n-2) \cup \CA(n-3)$:} 
Given $(c_1, \ldots, c_t) \in \CA(n-2)$, modify the last two parts by
\[ (c_{t-1}, c_t) \mapsto
\begin{cases} (c_{t-1}, c_t + 1, 1) & \text{if $t$ is odd} \\
& \quad \text{(image in $\CA^e(n)$ with $c_{2s-1} > 1$, $c_{2s} =1$),} \\
(c_{t-1}+1, c_t+1) & \text{if $t$ is even} \\
& \quad \text{(image in $\CA^e(n)$ with $c_{2s-1} > 1$, $c_{2s} > 1$)}, \end{cases} \]
where the first case includes the possibility that $t = 1$, i.e., there is no part $c_{t-1}$.
In the first case, the new final pair $(c_t + 1, 1)$ and any previous pairs satisfy the pairwise condition, so the image is in $\CA^e(n)$.  In the second case, the condition $c_{t-1} \ne c_t$ from $\CA(n-2)$ implies that $c_{t-1} + 1 \ne c_t + 1$, so the image is again in $\CA^e(n)$.

Given $(c_1, \ldots, c_t) \in \CA(n-3)$, modify the last part by
\[ c_t \mapsto
\begin{cases} (1, c_t + 2) & \text{if $t$ is odd} \\
& \quad \text{(image in $\CA^e(n)$ with $c_{2s-1} = 1$, $c_{2s} >2$),} \\
(c_t, 1, 2) & \text{if $t$ is even} \\
& \quad \text{(image in $\CA^e(n)$ with $c_{2s-1} = 1$, $c_{2s} = 2$)}. \end{cases} \]
Note that the first case produces a pair of unequal parts $(1, c_t + 2)$ at the end; any previous pairs satisfy the pairwise conditions for $\CA(n-3)$ so that the image is in $\CA^e(n)$.  In the second case, the pair with second part $c_t$ along with any previous pairs satisfy the pairwise conditions for $\CA(n-3)$ and the new final pair $(1,2)$ is valid, so the image is again in $\CA^e(n)$.

It is clear that these maps are inverses, establishing \eqref{k=1even}.
\end{proof}

See Table \ref{tabcabij} for the $n = 6$ example of the bijections in the proof.

\begin{table}[h]
\begin{center}
\begin{tabular}{r|r||r|r}
$\CA^o(6)$ & $\CA(5)$ & $\CA^e(6)$ & $\CA(4) \cup \CA(3)$ \\ \hline
6 & 5 & 51 & 4 \\
411 & 41 & 42 & 31 \\
321 & 32 & 24 & 13 \\
312 & 311 & 2121 & 211 \\
231 & 23 & 2112 & 21 \\
213 & 212 & 15 & 3 \\
141 & 14 & 1221 & 121 \\
132 & 131 & 1212 &  12 \\
123 & 122 
\end{tabular}
\end{center}
\caption{The bijections $\CA^o(6) \cong \CA(5)$ and $\CA^e(6) \cong \CA(4) \cup \CA(3)$ used in the proof of Theorem \ref{k=1}.} \label{tabcabij}
\end{table}

Theorem \ref{k=1} also follows from the generating function
\[ 1 + \sum_{n \ge 1} ca(n) x^n = \frac{1-x^2}{1-x-x^2-x^3}, \]
a special case of Theorem \ref{gflower} in the next section.

Another set of restricted compositions counted by \cite[A000213]{o} are those with no consecutive parts 1; see Table \ref{tabno11} for initial examples of these sets $C_{\widehat{1,1}}(n)$ and their counts.  These generalize compositions considered by Cayley where there are no parts 1; like Arndt's original compositions, Cayley's are also counted by the Fibonacci numbers.  See \cite{hoc, hout} for more historical background. 

\begin{table}[h]
\begin{center}
\begin{tabular}{c|l|c}
$n$ & $C_{\widehat{1,1}}(n)$ & $c_{\widehat{1,1}}(n)$ \\ \hline
1 & 1 & 1 \\
2 & 2 & 1 \\
3 & 3, 21, 12 & 3 \\
4 & 4, 31, 22, 13, 121  & 5 \\
5 & 5, 41, 32, 23, 221, 212, 14, 131, 122 & 9 \\
\end{tabular}
\end{center}
\caption{Compositions without adjacent parts 1 for small values of $n$.} \label{tabno11}
\end{table}

We establish the connection between the Carlitz--Arndt compositions of $n$ and these $C_{\widehat{1,1}}(n)$ combinatorially.

\begin{thm}\label{without-11-bij}
There is a bijection $\CA(n) \cong C_{\widehat{1,1}}(n)$.
\end{thm}

\begin{proof}
\textbf{Map from $\CA(n)$:} Consider a composition $c = (c_1, \dots, c_t) \in \CA(n)$. Scan from left to right for triples of the form $(c_{2i-1}, c_{2i}, c_{2i+1}) = (a,1,1)$; necessarily, $a \ge 2$. In the first such instance, make the replacement
\[ (a,1,1) \mapsto \begin{cases} ((a+2)/2, (a+2)/2) & \text{if $a$ is even,} \\ ((a+1)/2, (a+1)/2,1) & \text{if $a$ is odd,} \end{cases} \]
which preserves the composition sum, and scan the resulting composition again from left to right, making additional replacements as necessary.  
This procedure removes any instances of $(1,1)$, so the resulting composition is in $C_{\widehat{1,1}}(n)$. 

\textbf{Map from $C_{\widehat{1,1}}(n)$:} Given $d = (d_1, \dots, d_s) \in C_{\widehat{1,1}}(n)$, scan from left to right for triples of the form $(d_{2i - 1}, d_{2i}, d_{2i+1}) = (k,k,m)$; necessarily, $k \ge 2$ and we allow the case where $d_{2i}$ is the last part by considering $d_{2i+1} = 0$.  In the first such instance, make the replacement
\[ (k,k,m) \mapsto \begin{cases} (2k-1,1,1) & \text{if $m = 1$,} \\ (2k-2,1,1,m) & \text{otherwise,} \end{cases} \]
which preserves the composition sum, and scan the resulting composition again from left to right, making additional replacements as necessary.
The procedure removes any instances of $c_{2i-1} = c_{2i}$ (note that any adjacent parts 1 occur at $c_{2i}$ and $c_{2i+1}$ for some $i$), so the resulting composition is in $\CA(n)$.

It is straightforward to confirm that the two maps are inverses.
\end{proof}

As an example of the reverse map, consider $(2,2,2,2,2) \in C_{\widehat{1,1}}(10)$.  The first triple of the form $(k,k,m)$ starting at an odd position is $(2,2,2)$ at $i = 1$.  This is replaced by $(2,1,1,2)$ giving $(2,1,1,2,2,2)$.  Scanning this composition gives $(2,2,0)$ starting at $i = 5$ (the $(2,2,2)$ starting at $i = 4$ is not relevant) which is replaced by $(2,1,1)$ to give $(2,1,1,2,2,1,1) \in \CA(10)$.

Table \ref{tabcano11bij} shows the nonidentity cases of the bijection for $n = 8$.
Of the 57 compositions in $\CA(8)$, just 11 contain at least one triple of the form $(a,1,1)$ at an odd index; the other 46 have no $(1,1)$ subsequence.  

\begin{table}[h]
\begin{center}
\begin{tabular}{l|l}
$\CA(8)$ & $C_{\widehat{1,1}}(8)$ \\ \hline
$611$        & $44$ \\
$4112$      & $332$ \\
$3113$      & $2213$ \\
$2114$      & $224$ \\
$31121$    & $22121$ \\
$21131$    & $2231$ \\
$21122$    & $2222$ \\
$31211$    & $3122$ \\
$21311$    & $21221$ \\
$13211$    & $1322$ \\
$12311$    & $12221$
\end{tabular}
\end{center}
\caption{The compositions that change in the bijection $\CA(8) \cong C_{\widehat{1,1}}(8)$ used in the proof of Theorem \ref{without-11-bij}.} \label{tabcano11bij}
\end{table}

\section{Absolute differences bounded below}

In order to generalize standard Carlitz compositions, recast the defining condition $c_i \ne c_{i+1}$ as $| c_i - c_{i+1}| \ge 1$.  Considering $| c_{2i-1} - c_{2i}| \ge k$ for a given positive integer $k$ leads to a family of compositions which we analyze in this section.  Motivations from the partition setting include Rogers--Ramanujan partitions, where parts differ by at least two (sometimes called super-distinct), and Schur partitions, which include the requirement that parts differ by at least three; see Andrews and Eriksson \cite{ae04} for more information.

\begin{definition} \label{calowerdef}
Given a positive integer $k$, let the generalized Carlitz--Arndt compositions $\CA_{\ge k}(n)$ be the compositions of $n$ such that $| c_{2i-1} - c_{2i}| \ge k$ for each $i$.  If the number of parts is odd, then there is no restriction on the last part.
\end{definition}

Table \ref{tabcalowerex} gives examples of $\CA_{\ge k}(n)$ for $k = 2$ and $k = 3$.  From the definition, it is clear that $\CA_{\ge k+1}(n) \subseteq \CA_{\ge k}(n)$ for all positive $k$.  Note that $\CA_{\ge 1}(n) = \CA(n)$ of the previous section.  Allowing $k = 0$ would give $\CA_{\ge 0}(n) = C(n)$, as there would be no restrictions.  

\begin{table}[h]
\begin{center}
\begin{tabular}{c|l|l}
$n$ & $\CA_{\ge 2}(n)$ & $\CA_{\ge 3}(n)$\\ \hline
1 & 1 & 1\\
2 & 2 & 2\\
3 & 3 & 3  \\
4 & 4, 31, 13  & 4 \\
5 & 5, 41, 311, 14, 131 & 5, 41, 14 \\
6 & 6, 51, 42, 411, 312, 24, 15, 141, 132 & 6, 51, 411, 15, 141 
\end{tabular}
\end{center}
\caption{The generalized Carlitz--Arndt compositions $\CA_{\ge k}(n)$ for small $n$ and $k = 2, 3$.} \label{tabcalowerex}
\end{table}

Several counts $ca_{\ge k}(n)$ are given in Table \ref{tabgek} along with the recurrence signature for each $k$; this shorthand gives the coefficients of terms in a homogeneous linear recurrence.  For example, the recurrence for $ca_{\ge 1}(n) = ca(n)$ shown in Theorem \ref{k=1} is denoted $[1,1,1]$.  We note that, of these sequences, only the $k=1$ row is currently in the On-Line Encyclopedia of Integer Sequences \cite{o}.

\begin{table}[h]
\begin{center}
\begin{tabular}{c|rrrrrrrrrr|l}
$k \backslash n$ & 1 & 2 & 3 & 4 & 5 & 6 & 7 & 8 & 9 & 10 & recurrence signature \\ \hline
1 & 1 & 1 & 3 & 5 & 9 & 17 & 31 & 57 & 105 & 193 & $[1,1,1]$ \\
2 & 1 & 1 & 1 & 3 & 5 & 9 & 13 & 23 & 37 & 65 & $[1,1,-1,2]$ \\
3 & 1 & 1 & 1 & 1 & 3 & 5 & 9 & 13 & 19 & 29 &  $[1,1,-1,0,2]$ \\
4 & 1 & 1 & 1 & 1 & 1 & 3 & 5 & 9 & 13 & 19 & $[1,1,-1,0,0,2]$ 
\end{tabular}
\end{center}
\caption{Sequences $ca_{\ge k}(n,k)$ for small $n$ and $k$ with recurrence signatures.} \label{tabgek}
\end{table}

We establish the general recurrence relation suggested in Table \ref{tabgek} combinatorially.

\begin{thm} \label{gek}
For a given positive integer $k$, the generalized Carlitz--Arndt compositions $\CA_{\ge k}(n)$ satisfy
\[ca_{\ge k}(n) = ca_{\ge k}(n-1) + ca_{\ge k}(n-2) - ca_{\ge k}(n-3) + 2ca_{\ge k}(n-k-2)\]
with initial values $ca_{\ge k}(1) = \cdots = ca_{\ge k}(k+1) = 1$ and $ca_{\ge k}(k+2) = 3$.
\end{thm}

\begin{proof}
When $k=1$, the recurrence reduces to \[ca_{\ge 1}(n-1) + ca_{\ge 1}(n-2) + ca_{\ge 1}(n-3)\] which is Theorem \ref{k=1}.  Therefore, we assume $k \ge 2$.

For the initial values, the only compositions in $\CA_{\ge k}(n)$ for $n \le k+1$ are the single part compositions $(n)$.  Also, $\CA_{\ge k}(k+2) = \{ (k+2), (k+1,1), (1,k+1)\}$.

For the recurrence, we establish a bijection\[\CA_{\ge k}(n) \cup \CA_{\ge k}(n-3) \cong \CA_{\ge k}(n-1) \cup \CA_{\ge k}(n-2) \cup 2\CA_{\ge k}(n-k-2)\] where the coefficient two denotes two copies of that set.

As in the proof of Theorem \ref{k=1}, the bijection is simplified by considering the odd and even length compositions in $\CA_{\ge k}(n)$ separately.  We show 
\begin{gather}
\CA_{\ge k}^o(n) \cong \CA_{\ge k}(n-1), \label{gekodd} \\
\CA_{\ge k}^e(n) \cup \CA_{\ge k}(n-3) \cong  \CA_{\ge k}(n-2) \cup 2\CA_{\ge k}(n-k-2). \label{gekeven}
\end{gather}

The same bijection used in Theorem \ref{k=1} to establish \eqref{k=1odd} gives \eqref{gekodd}, so we proceed to a bijection for \eqref{gekeven}.

\textbf{Map from $\CA_{\ge k}^e(n) \cup \CA_{\ge k}(n-3)$:} The details of the bijection depend on the parity of composition length.  Statements on the images pertain to both the sum and satisfying the pairwise difference criterion.

For a composition in $\CA_{\ge k}^e(n)$ with last part $c_{2j}$, make the substitution
 \[  (c_{2j-1}, c_{2j}) \mapsto 
 \begin{cases} (c_{2j-1} - 1, c_{2j} - 1) & \text{if $c_{2j-1} > 1$, $c_{2j} > 1$} \\ & \; 
 \text{(image in $\CA_{\ge k}^e(n-2)$),} \\
 c_{2j-1}-1 & \text{if $c_{2j-1} > 1$, $c_{2j} = 1$} \\ & \;
 \text{(image in $\CA_{\ge k}^o(n-2)$, last part $\ge k$),} \\
 c_{2j}-k-1 & \text{if $c_{2j-1} = 1$ (image in $\CA_{\ge k}(n-k-2)$)}. \end{cases}\]
 Note that in the third case, the new part $c_{2j}-k-1$ could be zero.
 
 For a composition in $\CA_{\ge k}^e(n-3)$ with last part $c_{2j}$, make the replacement
 \[ c_{2j} \mapsto (c_{2j},1)\]
giving an image in $\CA_{\ge k}^o(n-2)$ with last part 1.
 
 For a composition in $\CA_{\ge k}^o(n-3)$ with last part $c_{2j-1}$, make the replacement
\[ c_{2j-1} \mapsto
 \begin{cases} c_{2j-1} - k + 1 & \text{if $c_{2j-1} \ge k - 1$ (image in $\CA_{\ge k}(n-k-2)$),} \\
 c_{2j-1} + 1 & \text{if $c_{2j-1} < k - 1$} \\ & \; \text{(image in $\CA_{\ge k}^o(n-2)$, $1 < $ last part $ < k$}). \end{cases}\]
 As before, in the first case, the new part $c_{2j-1} - k + 1$ could be zero.
 
 The notes on the results of the maps show that the images in $\CA_{\ge k}(n-2)$ are distinct while compositions in $\CA_{\ge k}(n-k-2)$ arise at most twice.
 
 \textbf{Map from $\CA_{\ge k}(n-2) \cup 2\CA_{\ge k}(n-k-2)$:} For a composition in $\CA_{\ge k}^e(n-2)$ with last part $c_{2j}$, make the substitution
\[(c_{2j-1}, c_{2j}) \mapsto (c_{2j-1} + 1, c_{2j} + 1)\]
giving an image in $\CA_{\ge k}^e(n)$ with penultimate part greater than 1.

For a composition in $\CA_{\ge k}^o(n-2)$ with last part $c_{2j-1}$, apply
\[ c_{2j-1} \mapsto \begin{cases} 
(c_{2j-1}+1,1) & \text{if $c_{2j-1} > k-1$ \; (image in $\CA_{\ge k}^e(n)$, last part 1)} \\
c_{2j-1} - 1 & \text{if $1 < c_{2j-1} \le k-1$} \\ & \; \text{(image in $\CA_{\ge k}^o(n-3)$, last part $< k-1$)}, \\
\text{removed} & \text{if $c_{2j-1} = 1$ \; (image in $\CA_{\ge k}^e(n-3)$)}. \end{cases} \]

For a composition in the first copy of $\CA_{\ge k}^e(n-k-2)$ with last part $c_{2j}$, make the substitution
\[ c_{2j} \mapsto (c_{2j}, k-1) \]
giving an image in $\CA_{\ge k}^o(n-3)$ with last part $k-1$.

For a composition in the first copy of $\CA_{\ge k}^o(n-k-2)$ with last part $c_{2j-1}$, make the substitution
\[ c_{2j-1} \mapsto c_{2j-1} + k-1 \]
giving an image in $\CA_{\ge k}^o(n-3)$ with last part greater than $k - 1$.

For a composition in the second copy of $\CA_{\ge k}^e(n-k-2)$ with last part $c_{2j}$, make the substitution
\[  c_{2j} \mapsto (c_{2j}, 1,k+1)\] 
giving an image in $\CA_{\ge k}^e(n)$ with penultimate part 1 and last part $k+1$.

For a composition in the second copy of $\CA_{\ge k}^o(n-k-2)$ with last part $c_{2j-1}$, make the substitution
\[ c_{2j-1} \mapsto (1, c_{2j-1}+k+1)\]
giving an image in $\CA_{\ge k}^e(n)$ with penultimate part 1 and last part greater than $k+1$.

The notes on the results of the maps show that the images in $\CA_{\ge k}^e(n) \cup \CA_{\ge k}(n-3)$ are distinct.
 
 It is straightforward to verify in each case that the maps are inverses, establishing the bijection.
 \end{proof}

Table \ref{tabgekbij2} shows the bijections for $k = 3$ and $n = 10$, the smallest example where all cases arise.

\begin{table}[h]
\begin{center}
\begin{tabular}{r|r||r|r}
$\CA^o_{\ge 3}(10)$ & $\CA_{\ge 3}(9)$ & $\CA^e_{\ge 3}(10) \cup \CA_{\ge 3}(7)$ & $\CA_{\ge 3}(8) \cup 2\CA_{\ge 3}(5)$   \\ \hline
$10$ & $9$ & $91$ & $8$ \\
$811$ & $81$ & $82$ & $71$ \\
$721$ & $72$ & $73$ & $62$ \\
$712$ & $711$ & $4141$ & $413$ \\
$631$ & $63$ &  $4114$ & $41$\\
$622$ & $621$ & $37$ & $26$\\
$613$ & $612$ & $28$ & $17$ \\
$523$ & $522$ & $19$ & $5$ \\ 
$514$ & $513$ & $1441$ & $143$ \\
$415$ & $414$ & $1414$ & $14$ \\ \cline{3-3}
$361$ & $36$ & $7$ & $5$ \\ 
$271$ & $27$ & $61$ & $611$ \\
$262$ & $261$ & $52$ & $521$ \\
$253$ & $252$ & $511$ & $512$\\
$181$ & $18$ & $412$ & $41$ \\
$172$ & $171$ & $25$ & $251$ \\
$163$ & $162$ & $16$ & $161$ \\
$154$ & $153$ & $151$ & $152$\\
$145$ & $144$ & $142$ & $14$
\end{tabular}
\end{center}
\caption{The bijections $\CA^o_{\ge 3}(10) \cong \CA_{\ge 3}(9)$ (compare Table \ref{tabcabij}) and $\CA^e_{\ge 3}(10) \cup \CA_{\ge 3}(7) \cong \CA_{\ge 3}(8) \cup 2\CA_{\ge 3}(5)$ used in the proof of Theorem \ref{gek}.} \label{tabgekbij2}
\end{table}

Following the approach of Checa and Ram\'irez \cite{cr24}, we also establish a generating function for $ca_{\ge k}(n)$, providing alternative proofs of Theorems \ref{k=1} and \ref{gek}.

\begin{thm} \label{gflower}
Given a positive integer $k$, the generating function
\[ \sum_{n \ge 0} ca_{\ge k}(n) x^n = \frac{1-x^2}{1-x-x^2+x^3-2x^{k+2}}.\]
\end{thm}

\begin{proof}
Each pair $(c_{2i-1}, c_{2i})$ in a composition of $\CA_{\ge k}(n)$ has a greater part $j \ge k+1$, a lesser part ranging from one to $j-k$, and two possible orders, thus a generating function is
\begin{equation}
\sum_{j \ge k+1} 2x^j(x + x^2 + \cdots + x^{j-k}) = \sum_{j \ge k+1} 2x^{j+1} \frac{1-x^{j-k}}{1-x} = \frac{2x^{k+2}}{(1-x)(1-x^2)} \label{gflowerpair}
\end{equation}
so that the compositions of even length have generating function
\[ \sum_{m \ge 0} \left( \frac{2x^{k+2}}{(1-x)(1-x^2)} \right)^{\!m} = \frac{1}{1-\frac{2x^{k+2}}{(1-x)(1-x^2)}} = \frac{1-x-x^2+x^3}{1-x-x^2+x^3-2x^{k+2}}\]
and the generating functions for compositions of odd length have the same expression using \eqref{gflowerpair} along with $x/(1-x)$ for the unrestricted final term:
\[ \sum_{m \ge 0} \left( \frac{2x^{k+2}}{(1-x)(1-x^2)} \right)^{\!m} \left( \frac{x}{1-x} \right) = \frac{x-x^3}{1-x-x^2+x^3-2x^{k+2}}.\]
Combining these two parity-based generating functions gives the desired expression.
\end{proof}

\begin{remark}
One could modify Definition \ref{calowerdef} to require that odd-length compositions satisfy that the last part is at least $k$, in effect making the length of every composition even with a possible part 0 at the end, and applying the inequality to that final pair.  We state without proof that this would change the initial values but not the recurrence of Theorem \ref{gek}.  While this would reduce the counts in Table \ref{tabgek}, it would also limit the motivating connection to Rogers--Ramanujan and Schur partitions.
\end{remark}

It was known to ancient Sanskrit prosodists that compositions with parts restricted to $\{1,2\}$ are counted by what are now known as the Fibonacci numbers with recurrence $F_n = F_{n-1} + F_{n-2}$ \cite{s86}.  Analogously, compositions with parts $\{1, 1', 2\}$, where a part $1'$ contributes one to the sum but is considered distinct from a part 1, are counted by the Pell numbers with recurrence $P_n = 2P_{n-1} + P_{n-2}$ \cite[A000129]{o}.  We connect the generalized Carlitz--Arndt compositions $\CA_{\ge k}(n)$ and a certain subset of these Pell compositions.

\begin{definition}
Let $P_{\ge k}(n)$ be the compositions with parts $\{1, 1', 2\}$ whose sum is $n$ that satisfy
\begin{enumerate}
\item[(Pa)] The first part is either 1 or $1'$ (not 2).
\item[(Pb)] Neither subsequence $(1, 1')$ or $(1', 1)$ occurs.
\item[(Pc)] Except at the end, each run of parts 1 or parts $1'$ has length at least $k$.
\item[(Pd)] The last part is either 1 or 2 (not $1'$).
\end{enumerate}
A run refers to a subsequence of equal parts.  The effect of condition (Pb) is that any run of ones is all parts $1$ or all parts $1'$, not a mix.
\end{definition}

Table \ref{pex} gives examples of $P_{\ge k}(n)$ for $k = 1$ and $k = 2$.  From the definition, it is clear that $P_{\ge k+1}(n) \subseteq P_{\ge k}(n)$ for all positive $k$.

\begin{table}[h]
\begin{center}
\begin{tabular}{c|l|l}
$n$ & $P_{\ge 1}(n)$ & $P_{\ge 2}(n)$\\ \hline
1 & 1 & 1 \\
2 & 11 & 11 \\
3 & 12, $1'2$, 111 & 111  \\
4 & 121, $1'21$, 112, $1'1'2$, 1111  & 112, $1'1'2$, 1111
\end{tabular}
\end{center}
\caption{The restricted Pell composition $P_{\ge k}(n)$ for small $n$ and $k = 1, 2$.}  \label{pex}
\end{table}

We connect the generalized Carlitz--Arndt compositions $\CA_{\ge k}(n)$ and the restricted Pell compositions $P_{\ge k}(n)$.

\begin{thm} \label{aclowerP}
For each positive integer $k$, there is a bijection $\CA_{\ge k}(n) \cong P_{\ge k}(n)$.
\end{thm}

\begin{proof}
\textbf{Map from $\CA_{\ge k}(n)$:} Given a composition $c = (c_1, \ldots, c_t) \in \CA_{\ge k}(n)$, we map each pair $(c_{2i-1}, c_{2i})$ (and single $c_t$ if the length is odd) into a subsequence of parts from $\{1, 1', 2\}$ satisfying the conditions (Pa)--(Pd) with sum $c_{2i-1} + c_{2i}$ (or $c_t$).

For each $(c_{2i-1}, c_{2i}) = (a,b)$, the defining condition $| a - b| \ge k$ has two cases, with $a$ or $b$ the greater part.  Using superscripts to denote repetition of parts, make the replacement
\[ (a,b) \mapsto \begin{cases} (1^{a - b}, 2^b) & \text{if $a - b \ge k$}, \\ ((1')^{b-a}, 2^a) & \text{if $b - a \ge k$}. \end{cases} \]
Since $a - b + 2b = a + b$ and analogously for the second case, the pair $(a, b)$ and its image have the same sum.

If the composition length $t$ is odd, make the replacement \[c_t \mapsto 1^{c_t}\] which clearly preserves the sum.

The image of $c$ is then the ordered concatenation of images of $(c_{2i-1}, c_{2i})$ and, if the length is odd, the image of the last part $c_t$.  Since each map above preserves the sum of the pairs (and final singleton for odd-length $c$) and the conditions (Pa)--(Pd) are met, the image is in  $P_{\ge k}(n)$.

\textbf{Map from $P_{\ge k}(n)$:} Given $p \in P_{\ge k}(n)$, partition it into subsequences of the form $(1^g, 2^h)$ or $((1')^g,2^h)$ where $g \ge k$ by condition (Pc) with a possible final subsequence $1^f$ for any $f \ge 1$.  Make the substitutions
\begin{align*}
(1^g, 2^h) & \mapsto (g+h, h), \\
((1')^g,2^h) & \mapsto (h, g+h), \\
1^f & \mapsto f,
\end{align*}
which all preserve the sum.  In the first two cases, the absolute difference of the two parts is at least $k$.  The image of $p$ is then the ordered concatenation of images of these subsequences, and the image is in $\CA_{\ge k}(n)$.  

The maps are clearly inverses, establishing the bijection.
\end{proof}

Table \ref{tabcagreaterPbij} shows the bijection for $k = 2$ and $n = 6$.

\begin{table}[h]
\begin{center}
\begin{tabular}{l|l}
$\CA_{\ge 2}(6)$ & $P_{\ge 2}(6)$ \\ \hline
$6$        & $1^6$ \\
$51$      & $1^42$ \\
$42$      & $1122$ \\
$411$      & $11121$ \\
$312$    & $11211$ \\
$24$    & $1'1'22$ \\
$15$    & $(1')^42$ \\
$141$    & $1'1'1'21$ \\
$132$    & $1'1'211$ 
\end{tabular}
\end{center}
\caption{The bijection $\CA_{\ge 2}(6) \cong P_{\ge 2}(6)$ used in the proof of Theorem \ref{aclowerP}.} \label{tabcagreaterPbij}
\end{table}

Figure \ref{figcagreater} provides a visualization of the bijection of Theorem \ref{aclowerP} using the bargraph representation of a composition, where each part $c_i$ corresponds to a column of $c_i$ boxes.  Parts 1 correspond to squares atop the left-hand column in a pair of columns (separated with thick lines) while parts $1'$ correspond to squares atop the right-hand column in a pair.  The image is read top-to-bottom in each pair.  This helps explain conditions (Pb) and (Pd), especially.  Compare the authors' \cite[Figure 1]{ht22} for the case of parts from $\{1, 2\}$.

\begin{figure}[h]
\begin{center}
\setlength{\unitlength}{.5cm}
\begin{picture}(13,6)
\thicklines
\put(0,0){\line(0,1){5}}
\put(1,0){\line(0,1){5}}
\put(2,0){\line(0,1){2}}
\put(3,0){\line(0,1){3}}
\put(4,0){\line(0,1){3}}
\put(5,0){\line(0,1){1}}

\put(0,5){\line(1,0){1}}
\put(0,4){\line(1,0){1}}
\put(0,3){\line(1,0){1}}
\put(0,2){\line(1,0){2}}
\put(0,1){\line(1,0){5}}
\put(0,0){\line(1,0){5}}
\put(3,3){\line(1,0){1}}
\put(3,2){\line(1,0){1}}

\put(9,2){\line(0,1){3}}
\put(11,1){\line(0,1){2}}
\put(13,0){\line(0,1){1}}
{\linethickness{1mm}
\put(8,0){\line(0,1){5}}
\put(10,0){\line(0,1){2}}
\put(12,0){\line(0,1){3}}}

\put(8,5){\line(1,0){1}}
\put(8,4){\line(1,0){1}}
\put(8,3){\line(1,0){1}}
\put(8,2){\line(1,0){2}}
\put(8,1){\line(1,0){5}}
\put(8,0){\line(1,0){5}}
\put(11,3){\line(1,0){1}}
\put(11,2){\line(1,0){1}}

\put(8.3,4.3){1}
\put(8.3,3.3){1}
\put(8.3,2.3){1}
\put(8.8,1.3){2}
\put(8.8,0.3){2}
\put(11.3,2.3){$1'$}
\put(11.3,1.3){$1'$}
\put(10.8,0.3){2}
\put(12.3,0.3){1}

\end{picture}
\end{center}
\caption{Representations of $(5,2,1,3,1) \in \CA_{\ge 2}(12)$ and the corresponding $(1,1,1,2,2,1',1',2,1) \in P_{\ge 2}(12)$ by the bijection of Theorem \ref{aclowerP}.}  \label{figcagreater}
\end{figure}
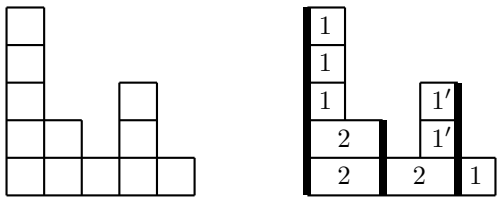

\begin{remark}
Composing the bijections of Theorem \ref{without-11-bij} and the $k = 1$ case of Theorem \ref{aclowerP} shows that $C_{\widehat{1,1}}(n) \cong CA_{\ge 1}(n)$.  The interested reader may want to detail the direct connection between these two sets of restricted compositions of $n$.
\end{remark}

\begin{remark}
With the correspondence between $ca_{\ge k}(n)$ and $p_{\ge k}(n)$ established by Theorem \ref{aclowerP}, one could establish the recurrence of Theorem \ref{gek} using the restricted Pell compositions $P_{\ge k}(n)$ rather than the generalized Carlitz--Arndt compositions $\CA_{\ge k}(n)$.  Having done this ourselves, we note that establishing 
\[p_{\ge k}(n) = p_{\ge k}(n-1) + p_{\ge k}(n-2) - p_{\ge k}(n-3) + 2p_{\ge k}(n-k-2)\]
combinatorially is comparable to the work involved in the proof of Theorem \ref{gek}.
\end{remark}

\section{Absolute differences bounded above}

A natural variation of the $\CA_{\ge k}(n)$ compositions of the previous section reverses the inequality.  In the general case, setting the condition $| c_i - c_{i+1} |  \le 1$ gives the gap-free compositions studied by Hitczenko and Knopfmacher \cite{hk05}.  Grabner and Knopfmacher considered the analogous gap-free partitions \cite{gk06}; Andrews calls these compact partitions \cite{a17} and traces the idea to Fine \cite{f88}.

\begin{definition} \label{caupperdef}
Given a positive integer $k$, let the generalized Carlitz--Arndt compositions $\CA_{\le k}(n)$ be the compositions of $n$ such that $| c_{2i-1} - c_{2i}| \le k$ for each $i$.  If the number of parts is odd, then there is no restriction on the last part.
\end{definition}

Table \ref{tabcaupperex} gives examples of $\CA_{\le k}(n)$ for $k = 1$ and $k = 2$.  Note that, of the examples shown, only $\CA_{\le 1}(4)$ is not the list of all compositions of the given $n$.  From the definition, it is clear that $\CA_{\le k}(n) \subseteq \CA_{\le k+1}(n)$ for all positive integers $k$.  

\begin{table}[h]
\begin{center}
\begin{tabular}{c|l|l}
$n$ & $\CA_{\le 1}(n)$ & $\CA_{\le 2}(n)$ \\ \hline
1 & 1 & 1\\
2 & 2, 11 & 2, 11 \\
3 & 3, 21, 12, 111 & 3, 21, 12, 111  \\
4 & 4, 22, 211, 121, 112, $1^4$  & 4, 31, 22, 211,13, 121, 112, $1^4$ 
\end{tabular}
\end{center}
\caption{The generalized Carlitz--Arndt compositions $\CA_{\le k}(n)$ for small $n$ and $k = 1, 2$.} \label{tabcaupperex}
\end{table}

Several counts $ca_{\le k}(n)$ are given in Table \ref{tablek} along with a recurrence signature for each $k$.  We note that, of these sequences, only the $k=1$ row is currently in the On-Line Encyclopedia of Integer Sequences \cite[A107383]{o} where it does not currently have a combinatorial interpretation.

\begin{table}[h]
\begin{center}
\begin{tabular}{c|rrrrrrrrrr|l}
$k \backslash n$ & 1 & 2 & 3 & 4 & 5 & 6 & 7 & 8 & 9 & 10 & recurrence signature \\ \hline
1 & 1 & 2 & 4 & 6 & 12 & 20 & 36 & 64 & 112 & 200 & $[0,2,2]$ \\
2 & 1 & 2 & 4 & 8 & 14 & 28 & 52 & 100 & 188 & 360 & $[0,2,2,2]$  \\
3 & 1 & 2 & 4 & 8 & 16 & 30 & 60 & 116 & 228 & 444 &  $[0,2,2,2,2]$ \\
4 & 1 & 2 & 4 & 8 & 16 & 32 & 62 & 124 & 244 & 484 & $[0,2,2,2,2,2]$
\end{tabular}
\end{center}
\caption{Sequences $ca_{\le k}(n,k)$ for small $n$ and $k$ with recurrence signatures.} \label{tablek}
\end{table}

Analogously to Theorem \ref{gflower} (with fewer details), we determine generating functions for $ca_{\le k}(n)$ which establish both the general recurrence relations suggested in Table \ref{tablek} and a family of higher degree recurrence relations with fewer nonzero terms.

\begin{thm} \label{gfupper}
Given a positive integer $k$, the generating function
\[ \sum_{n \ge 0} ca_{\le k}(n) x^n = \frac{1-x^2}{1-x-2x^2+2x^{k+3}} = \frac{1+x}{1-2x^2 - 2x^3 - \cdots - 2x^{k+2}}.\]
\end{thm}

\begin{proof}
Each pair $(c_{2i-1}, c_{2i})$ in a composition of $\CA_{\le k}(n)$ has either two equal parts $j$ or a lesser part $j$ and a greater part ranging from $j+1$ to $j+k$ in two possible orders, thus a generating function is
\[ \sum_{j \ge 1} x^{2j} + 2x^j(x^{j+1} + x^{j+2} + \cdots + x^{j+k}) = \frac{x^2+x^3-2x^{k+3}}{(1-x)(1-x^2)} \]
so that the compositions of even length have generating function
\[ \sum_{m \ge 0} \left(  \frac{x^2+x^3-2x^{k+3}}{(1-x)(1-x^2)} \right)^{\!m} = \frac{1-x-x^2+x^3}{1-x-2x^2+2x^{k+3}}\]
and the compositions of odd length have generating function
\[ \sum_{m \ge 0} \left(  \frac{x^2+x^3-2x^{k+3}}{(1-x)(1-x^2)} \right)^{\!m} \left( \frac{x}{1-x} \right) = \frac{x-x^3}{1-x-2x^2+2x^{k+3}}.\]
Combining these two parity-based generating functions gives the first rational expression of the theorem statement; the second comes from removing the factor $(1-x)$ from the numerator and denominator.
\end{proof}

\begin{remark}
Either of the two linear recurrences of Theorem \ref{gfupper} might be preferable depending on the context.  While
\[ ca_{\le k}(n) = 2ca_{\le k}(n-2) + 2ca_{\le k}(n-3) + \cdots + 2ca_{\le k}(n-k-2)\]
has the lower degree, 
\[ ca_{\le k}(n) = ca_{\le k}(n-1) + 2ca_{\le k}(n-2) - 2ca_{\le k}(n-k-3) \]
has fewer nonzero coefficients, a characteristic that can simplify combinatorial proofs, as we will see in Theorem \ref{lek}.
\end{remark}

Similar to the $P_{\ge k}(n)$ compositions of the previous section using parts $\{1, 1', 2\}$, we consider another type of composition that we will show is equinumerous with $\CA_{\le k}(n)$.

\begin{definition}
Let $Q_{\le k}(n)$ be the compositions with parts $\{1, 1', 2, 4, 6, \ldots \}$ whose sum is $n$ that satisfy
\begin{enumerate}
\item[(Qa)] Neither subsequence $(1, 1')$ or $(1', 1)$ occurs.
\item[(Qb)] Except at the end, each run of parts 1 or parts $1'$ has length at most $k$.
\item[(Qc)] The last part is either 1 or an even number (not $1'$).
\end{enumerate}
\end{definition}

Table \ref{qex} gives examples of $Q_{\le k}(n)$ for $k = 1$ and $k = 2$.  From the definition, it is clear that $Q_{\le k}(n) \subseteq Q_{\le k+1}(n)$ for all positive integers $k$.

\begin{table}[h]
\begin{center}
\begin{tabular}{c|l|l}
$n$ & $Q_{\le 1}(n)$ & $Q_{\le 2}(n)$\\ \hline
1 & 1 & 1 \\
2 & 2, 11 & 2, 11 \\
3 & 21, 12, $1'2$, 111 & 21, 12, $1'2$, 111  \\
4 & 4, 22, 211, 141, $1'41$, 1111  & 4, 22, 211, 141, $1'41$, 112, $1'1'2$, 1111 
\end{tabular}
\end{center}
\caption{The compositions $Q_{\le k}(n)$ for small $n$ and $k = 1, 2$.} \label{qex}
\end{table}

We connect the generalized Carlitz--Arndt compositions $\CA_{\le k}(n)$ and $Q_{\le k}(n)$.

\begin{thm} \label{acupperQ}
For each positive integer $k$, there is a bijection $\CA_{\le k}(n) \cong Q_{\le k}(n)$.
\end{thm}

\begin{proof}
\textbf{Map from $\CA_{\le k}(n)$:} Given a composition $c = (c_1, \ldots, c_t) \in \CA_{\le k}(n)$, we map each pair $(c_{2i-1}, c_{2i})$ (and single $c_t$ if the length is odd) into a subsequence of parts from $\{1, 1', 2, 4, \ldots\}$ satisfying the conditions (Qa)--(Qc) with sum $c_{2i-1} + c_{2i}$ (or $c_t$).

For each $(c_{2i-1}, c_{2i}) = (a,b)$, the defining condition $| a - b| \le k$ has three cases, with $a$ or $b$ the greater part or $a=b$.  Using superscripts to denote repetition of parts, make the replacement
\[ (a,b) \mapsto \begin{cases} (1^{a - b}, 2b) & \text{if $a > b$}, \\ ((1')^{b-a}, 2^a) & \text{if $a < b$}, \\ 2a & \text{if $a=b$}. \end{cases} \]
In each case, the pair $(a, b)$ and its image have the same sum.

If the composition length $t$ is odd, make the replacement \[c_t \mapsto 1^{c_t}\] which clearly preserves the sum.

The image of $c$ is then the ordered concatenation of images of $(c_{2i-1}, c_{2i})$ and, if the length is odd, the image of the last part $c_t$.  Since each map above preserves the sum of the pairs (and final singleton for odd-length $c$) and the conditions (Qa)--(Qc) are met, the image is in  $Q_{\le k}(n)$.

\textbf{Map from $Q_{\le k}(n)$:} Given $q \in Q_{\le k}(n)$, partition it into subsequences of the form $(1^g, 2h)$ or $((1')^g,2h)$, where $1 \le g \le k$ by condition (Qb), or $(2h)$, with a possible final subsequence $1^f$ for any $f \ge 1$.  Make the substitutions
\begin{align*}
(1^g, 2h) & \mapsto (g+h, h), \\
((1')^g,2h) & \mapsto (h, g+h), \\
(2h) & \mapsto (h,h) \\
1^f & \mapsto f,
\end{align*}
which all preserve the sum.  In the first two cases, the absolute difference of the two parts is at most $k$.  The image of $q$ is then the ordered concatenation of images of these subsequences, and the image is in $\CA_{\le k}(n)$.  The maps are clearly inverses.
\end{proof}

Table \ref{tabcaupperQbij} shows the bijection for $k = 1$ and $n = 5$.

\begin{table}[h]
\begin{center}
\begin{tabular}{l|l}
$\CA_{\le 1}(5)$ & $Q_{\le 1}(5)$ \\ \hline
$5$        & $1^5$ \\
$32$		& $14$ \\
$23$		& $1'4$ \\
$221$      & $41$ \\
$212$	& $1211$ \\
$2111$	& $122$ \\
$122$	& $1'211$ \\
$1211$	& $1'22$ \\
$113$	& $2111$ \\
$1121$	& $212$ \\
$1112$	& $21'2$ \\
$1^5$    & $221$ 
\end{tabular}
\end{center}
\caption{The bijection $\CA_{\le 1}(5) \cong Q_{\le 1}(5)$ used in the proof of Theorem \ref{acupperQ}.} \label{tabcaupperQbij}
\end{table}

Figure \ref{figcaupper} provides a visualization of the bijection of Theorem \ref{acupperQ} using the bargraph representation.  As in Figure \ref{figcagreater}, parts 1 are squares atop the left-hand column in a pair, parts $1'$ are squares atop the right-hand column in a pair, and the image is read top-to-bottom in each pair of columns.  Again, the graphic supports the defining conditions of $Q_{\le k}(n)$.

\begin{figure}[h]
\begin{center}
\setlength{\unitlength}{.5cm}
\begin{picture}(15,4)
\thicklines
\put(0,0){\line(0,1){1}}
\put(1,0){\line(0,1){2}}
\put(2,0){\line(0,1){4}}
\put(3,0){\line(0,1){4}}
\put(4,0){\line(0,1){3}}
\put(5,0){\line(0,1){1}}
\put(6,0){\line(0,1){4}}
\put(7,0){\line(0,1){4}}

\put(0,1){\line(1,0){7}}
\put(0,0){\line(1,0){7}}
\put(1,2){\line(1,0){1}}
\put(2,2){\line(1,0){2}}
\put(2,4){\line(1,0){1}}
\put(2,3){\line(1,0){2}}
\put(3,2){\line(1,0){1}}
\put(6,2){\line(1,0){1}}
\put(6,3){\line(1,0){1}}
\put(6,4){\line(1,0){1}}

\put(11,1){\line(0,1){1}}
\put(13,3){\line(0,1){1}}
\put(17,0){\line(0,1){4}}
{\linethickness{1mm}
\put(10,0){\line(0,1){1}}
\put(12,0){\line(0,1){4}}
\put(14,0){\line(0,1){3}}
\put(16,0){\line(0,1){4}}}

\put(10,1){\line(1,0){2}}
\put(10,0){\line(1,0){7}}
\put(11,2){\line(1,0){1}}
\put(12,3){\line(1,0){2}}
\put(12,4){\line(1,0){1}}
\put(14,1){\line(1,0){3}}
\put(16,2){\line(1,0){1}}
\put(16,3){\line(1,0){1}}
\put(16,4){\line(1,0){1}}

\put(10.8,0.3){2}
\put(12.3,3.3){1}
\put(12.8,1.3){6}
\put(14.8,0.3){2}
\put(11.3,1.3){$1'$}
\put(16.3,3.3){1}
\put(16.3,2.3){1}
\put(16.3,1.3){1}
\put(16.3,0.3){1}

\end{picture}
\end{center}
\caption{Representations of $(1,2,4,3,1,1,4) \in \CA_{\le 1}(16)$ and the corresponding $(1',2,1,6,2,1,1,1,1) \in Q_{\le 1}(16)$ by the bijection of Theorem \ref{acupperQ}.}  \label{figcaupper}
\end{figure}
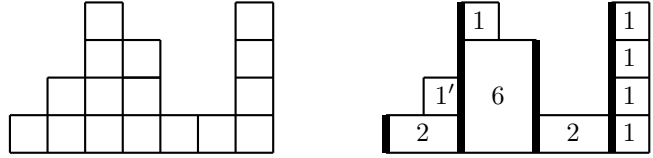

In the previous section, we established the $ca_{\ge k}(n)$ recurrence in Theorem \ref{aclowerP} directly from those generalized Carlitz--Arndt compositions rather than using the equinumerous compositions $P_{\ge k}(n)$.  Here, in contrast, we will use the $Q_{\le k}(n)$ compositions to establish combinatorially the recurrence for $ca_{\le k}(n)$.

\begin{thm} \label{lek}
For a given positive integer $k$, the compositions $\CA_{\le  k}(n)$ satisfy
\[ca_{\le k}(n) = ca_{\le k}(n-1) + 2ca_{\le k}(n-2) - 2ca_{\le k}(n-k-3)\]
with initial values $ca_{\le k}(n) = 2^{n-1}$ for $1\le n \le k+2$ and $ca_{\le k}(k+3) = 2^{k+2}-2$.
\end{thm}

\begin{proof} 
For the initial values, recall that $c(n)=2^{n-1}$.  The claim is that, for $1\le n \le k+2$, every composition of $n$ is in $\CA_{\le  k}(n)$.  This is true because the maximum difference between a consecutive pair of parts in a composition of $k+2$ occurs with $(k+1,1)$ and $(1,k+1)$, which have absolute difference $k$; for smaller $n$, the maximum absolute difference is less than $k$.  For $n = k+3$, exactly the compositions $(k+2,1)$ and $(1,k+2)$ are excluded from $\CA_{\le k}(k+3)$ so that $ca_{\le k}(k+3) = 2^{k+2}-2$.

For the recurrence relation, we use the fact that $ca_{\le k}(n) = q_{\le k}(n)$ from Theorem \ref{acupperQ} and work with the $Q_{\le k}(n)$ compositions instead.
Similar to the technique in the proof of Theorem \ref{k=1}, we establish a slightly stronger result based on the last part of the compositions in $Q_{\le k}(n)$.  Let $Q^1_{\le k}(n)$ be the compositions of $Q_{\le k}(n)$ with last part 1, similarly $Q^e_{\le k}(n)$ for an even last part.  We demonstrate bijections
\begin{gather}
Q^1_{\le k}(n) \cong Q_{\le k}(n-1), \label{last1} \\
Q^e_{\le k}(n) \cup 2Q_{\le k}(n-k-3) \cong 2Q_{\le k}(n-2). \label{laste}
\end{gather}
Since $Q^1_{\le k}(n) \cap Q^e_{\le k}(n) = \varnothing$ and $Q^1_{\le k}(n) \cup Q^e_{\le k}(n) = Q_{\le k}(n)$ by property (Qc), these will imply
\[Q_{\le k}(n) \cup 2Q_{\le k}(n-k-3) \cong Q_{\le k}(n-1) \cup 2Q_{\le k}(n-2)\]
from which the result follows.

\textbf{First map, from $Q^1_{\le k}(n)$:} Given a composition in $Q^1_{\le k}(n)$, let $(q_{t-1},1)$ be the last two parts.  Modify those parts by
\[ (q_{t-1},1) \mapsto q_{t-1}.\] Since there is no $(1',1)$ subsequence by (Qa), the resulting composition of $n-1$ ends in 1 or an even number, i.e., satisfies (Qc).  Removing the last part cannot invalidate (Qa) or (Qb), so the image is in $Q_{\le k}(n-1)$.

\textbf{First map, from $Q_{\le k}(n-1)$:} Given a composition in $Q_{\le k}(n-1)$ with last part $q_s$, modify it by
\[ q_s \mapsto (q_s, 1)\]
so that the image satisfies (Qc) automatically.  Because the original compositions cannot end in a part $1'$, the map does not introduce a $(1',1)$ subsequence, so the image satisfies (Qa).  Since (Qb) does not restrict a terminal run of parts 1, that condition is also maintained, so the image is in $Q^1_{\le k}(n)$.

It is clear that these maps are inverses, establishing \eqref{last1}.

\textbf{Second map, from $Q^e_{\le k}(n) \cup 2Q_{\le k}(n-k-3)$:} Given a composition in $Q^e_{\le k}(n)$, write it as $(c, s^u, t)$ where we know $t$ is even and $s$ is the penultimate run (both $c$ and $s^u$ can be empty).  In the image, keep $c$ unchanged and modify $(s^u, t)$ by the rule
\[ (s^u, t) \mapsto 
\begin{cases} (s^u) & \text{if $t=2$ and $s = 1$ or an even number,} \\
(1^u) & \text{if $t=2$ and $s = 1'$,} \\
(s^u, t-2) & \text{if $t>2$}.
\end{cases}\]
Each case gives a composition of $n-2$.  No subsequences $(1,1')$ or $(1',1)$ are introduced (note that the second case changes the entire penultimate run of parts $1'$ to $1$) and no run lengths are increased, so (Qa) and (Qb) hold for the image.  Also, the image satisfies (Qc) with the change when $t = 2$ and $s = 1'$, so the image is in $Q_{\le k}(n-2)$ (with duplication).  Note that the first two cases give images where any terminal run of parts $1$ has length at most $k$, as it came from an internal run of parts $1$ or $1'$ in a composition of $Q^e_{\le k}(n)$ satisfying (Qb).

Given a composition in each copy of $Q_{\le k}(n-k-3)$, append $k+1$ parts 1 at the end.  The defining conditions are maintained so that the images are in $Q_{\le k}(n-2)$ (with duplication) with terminal runs of parts $1$ having length at least $k+1$.

\textbf{Second map, from $2Q_{\le k}(n-2)$:} Given a composition in the first copy of $Q_{\le k}(n-2)$, write it as $(c, t^v)$.  In the image, keep $c$ unchanged and modify $t^v$ by the rule
\[ t^v \mapsto 
\begin{cases} (t^v,2) & \text{if $t=1$ with $v \le k$ or $t \ge 2$,} \\
& \quad \text{(image in $Q^e_{\le k}(n)$ ending $(t,2)$ with $t = 1$ or even),} \\
1^{v-k-1} & \text{if $t=1$ with $v \ge k+1$ \; (image in  $Q_{\le k}(n-k-3)$).}
\end{cases}\]
Appending a part 2 or decreasing a terminal run of parts 1 maintains (Qa)--(Qc), so the images are in the described subsets.

Given a composition in the second copy of $Q_{\le k}(n-2)$, write it as $(c, t^v)$.  In the image, keep $c$ unchanged and modify $t^v$ by the rule
\[ t^v \mapsto 
\begin{cases} ((1')^v,2) & \text{if $t=1$ with $v \le k$} \\
& \quad \text{(image in $Q^e_{\le k}(n)$ ending $(1',2)$),} \\
1^{v-k-1} & \text{if $t=1$ with $v \ge k+1$ \; (image in  $Q_{\le k}(n-k-3)$),} \\
(t^{v-1},t+2) & \text{if $t \ge 2$} \\
& \quad \text{(image in $Q^e_{\le k}(n)$, last part even at least 4).}
\end{cases}\]
Again, the images are clearly in the described subsets.

As explained in the notes, there is no duplication from $2Q_{\le k}(n-2)$ among the images in $Q^e_{\le k}(n)$.  Each composition in $Q_{\le k}(n-k-3)$ appears twice, once from each copy of $Q_{\le k}(n-2)$.

It is straightforward to show that these maps are inverses, establishing \eqref{laste} and completing the proof.
\end{proof}

\begin{remark}
Note that the initial values are clear for the $\CA_{\le k}(n)$ compositions but not obvious for the equinumerous $Q_{\le k}(n)$ compositions.  Proving the $ca_{\le k}(n)$ recurrence using the generalized Carlitz--Arndt compositions would be comparable to the proof of Theorem \ref{aclowerP}.
\end{remark}

Table \ref{tabqbij} shows the bijections for $k = 1$ and $n = 6$.

\begin{table}[h]
\begin{center}
\begin{tabular}{r|r||r|r}
$Q^1_{\le 1}(6)$ & $Q_{\le 1}(5)$ & $Q^e_{\le 1}(6) \cup 2Q_{\le 1}(2)$ & $2Q_{\le 1}(4)$   \\ \hline
$411$ & $41$ & $6$ & $4$ \\
$2211$ & $221$ & $42$ & $4$ \\
$2121$ & $212$ & $24$ & $22$ \\
$21'21$ & $21'2$ & $222$ & $22$ \\
$21^4$ & $21^3$ &  $1212$ & $121$\\
$141$ & $14$ & $121'2$ & $121$\\
$1221$ & $122$ & $1'212$ & $1'21$  \\
$12111$ & $1211$ & $1'21'2$ & $1'21$ \\ \cline{3-3}
$1'41$ & $1'4$ & $2$ & $211$ \\
$1'221$ & $1'22$ & $11$ & $1^4$ \\ \cline{3-3}
$1'21^3$ & $1'211$ & $2$ & $211$ \\ 
$1^6$ & $1^5$ & $11$ & $1^4$ \\
\end{tabular}
\end{center}
\caption{The bijections $Q^1_{\le 1}(6) \cong Q_{\le 1}(6)$ and $Q^e_{\le 1}(6) \cup 2Q_{\le 2}(2) \cong 2Q_{\le 1}(4)$ used in the proof of Theorem \ref{lek}.} \label{tabqbij}
\end{table}

\end{document}